\documentclass{amsart}

\usepackage{color,latexsym,amsmath,amssymb,amsfonts,amsthm}

\usepackage{color,latexsym,amsmath,amssymb,amsfonts,amsthm,bm}
\usepackage{amscd}
\usepackage{graphicx}
\usepackage{tikz}
\usepackage{bm}

\newtheorem{lemma}{Lemma}[section]
\newtheorem{theorem}[lemma]{Theorem}
\newtheorem{proposition}[lemma]{Proposition}
\newtheorem{corollary}[lemma]{Corollary}
\newtheorem{first argument}[lemma]{First Argument}

\theoremstyle{definition}
\newtheorem{definition}[lemma]{Definition}
\theoremstyle{definitions}

\theoremstyle{remark}
\newtheorem{example}[lemma]{Example}
\newtheorem{examples}[lemma]{Examples}

\theoremstyle{remark}

\def\SC{\mathbb C}    
\def\SH{\mathbb H}    
\def\SN{\mathbb N}    
\def\SQ{\mathbb Q}    
\def\SZ{\mathbb Z}    

\def\mm{\mathfrak m}    

\def\Aa{{\mathcal A}}           
\def\Bb{{\mathcal B}}           
 
\def\Mm{{\mathcal M}}

\def\Oo{{\mathcal O}}           

\def\u2{\underline{2}}

\def\int{\mathrm{Int}}         

\def\spec1{\mathrm{Spec}^1}

\def\be{\begin{equation}}
\def\ee{\end{equation}}

\begin{document}

\title[Group rings and Pr\"ufer]{When is the ring of integer-valued polynomials over a group ring a Pr\"ufer domain?}

\begin{abstract}
In their study of the ring of integer-valued polynomials in noncommutative algebra, Peruginelli and Werner characterized the algebras for which this ring is a Pr\"ufer domain. Here, we apply their results to the case of group algebras.
\end{abstract}

\author{Jean-Luc Chabert}
\address{Facult\'e des Sciences\\
Universit\'e de Picardie\\
80039 Amiens, France\\
LAMFA CNRS-UMR 7352}

\email{jean-luc.chabert@u-picardie.fr}


\maketitle


\section{Introduction}

In all the paper, $D$ denotes a commutative integral domain distinct from its quotient field $K$.

\begin{definition}[only for this paper]
A $D$-algebra $\Aa$ is said to be a {\em nice $D$-algebra} if it is unitary, is finitely generated and torsion-free as a $D$-module and is such that $\Aa\cap K=D$.
\end{definition}

In the sequel, $\Aa$ always denotes a nice $D$-algebra. Let $\Bb=K\otimes_D\Aa$ and $m=\dim_K \Bb$. With these hypotheses, we may assume that $\Bb$ contains $K$ and $\Aa$.

\begin{examples} The following rings are nice $D$-algebras.

a) $\Mm_n(D)$ the ring of square matrices of size $n$ with entries in $D$.

b) $\SH_D$ the ring of Hamilton quaternions with coefficients in $D$.

c) $D[G]$ the group ring where $G$ is a finite group.	
\end{examples}

In a recent paper, Peruginelli and Werner characterized the case where the following ring
 of {\em integer-valued polynomials on} $\Aa$ is a Pr\"ufer domain:
 $$\int_K(\Aa)=\{f(X)\in K[X]\mid f(\Aa)\subseteq\Aa\}.$$

\begin{theorem}\cite[Theorem 1.7\,(1)] {bib:PW2026}\label{th:B}
Assume that $\Aa$ is a nice $D$-algebra. Then, the ring $\int_K(\Aa)$ is a Pr\"ufer domain if and only if $D$ itself is a Pr\"ufer domain and $\Aa=\Aa'$ where $\Aa'$ is the set formed by the elements of $\Bb$ which are integral over $D$.
\end{theorem}

Applying their results to the first two examples of nice algebras, they show that a matrix algebra $\Mm_n(D)$ where $n\geq 2$ never leads to a Pr\"ufer domain (\cite[\S\,4]{bib:LW2012} or \cite[Corollary 3.4]{bib:PW2016}), whereas the ring $\int_{\SQ}(\SH_{\SZ_{(2)}})$ is a Pr\"ufer domain where $\SH_{\SZ_{(2)}}$ denotes the Hurwitz quaternion ring on $\SZ_{(2)}$ \cite[Theorem 5.4]{bib:PW2026}, that is,
$$\SH_{\SZ_{(2)}}=\left\{a_0+a_1{\bf i}+a_2{\bf j}+a_3{\bf k}\mid a_i\in \SZ_{(2)}\;\forall i \textrm{ or } a_i\in \SZ_{(2)}+\frac{1}{2}\;\forall i\right\}.$$


When the Jacobson radical $J(D)$ of $D$ is equal to $(0)$ and $\int_K(\Aa)$ is a Pr\"ufer domain, then the $D$-algebra $\Aa$ is necessarily commutative and we have a more precise result:

\begin{theorem}\cite[Theorem 1.7\,(2)]{bib:PW2026}\label{th:14}
Assume that that the Jacobson radical $J(D)$ of $D$ is $(0)$ and that $\int(D)$ is a Pr\"ufer domain. Then, $\int_K(\Aa)$ is a Pr\"ufer domain if and only if 

\centerline{$\Aa\simeq \prod_{i=1}^t \Aa_i$} 

\smallskip

\noindent where $\Aa_i$ is the integral closure of $D$ in a finite extension $F_i$ of $K$. In this case, each $\Aa_i$ is a commutative integral domain such that $\int_{F_i}(\Aa_i)$ is a Pr\"ufer domain.
\end{theorem}


\section{Group rings}

Let $G$ be a group and let $D[G]$ be the $D$-algebra  
$$D[G]=\{\sum_{g\in G}a_g\,g\;\big\vert \; a_g\in D \textrm{  almost all zero} \}.$$ 
It is a free $D$-module with basis $\{g\}_{g\in G}$ and a multiplication defined by $$\left(\sum_{h\in G}a_hh\right)\times\left(\sum_{k\in G}b_kk\right)=\sum_{g\in G}\left(\sum_{hk=g}a_hb_k\right)g.$$
The group ring  $D[G]$ is a nice $D$-algebra if and only if $G$ is a finite group. Thus, from now on we assume that $G$ is a finite group. Of course, if $\Aa=D[G]$, then $\Bb=K[G]$ and 

\centerline{$m=\dim_K K[G]=\mathrm{rk}_D (D[G])=\vert G\vert$.}

\smallskip

\noindent By Theorem~\ref{th:14}, if $J(D)=(0)$ and if $\int_K(D[G])$ is a Pr\"ufer domain, then $D[G]$ is commutative. Clearly, $D[G]$ is commutative if and only if $G$ is abelian.

\medskip

The purpose of this text is to characterise the finite groups $G$ such that 
$$\int_{K}(\Oo_K[G])=\{f\in K[X]\mid f(\Oo_K[G])\subseteq \Oo_K[G]\}$$
is a Pr\"ufer domain, where $K$ denotes an algebraic number field and $\Oo_K$ its ring of integers. As previously said, $G$ has to be abelian. We are going to prove that this necessary condition is also sufficient.

\begin{example}
Let $C_p$ be a cyclic group whose order is a prime number $p$ and let $\zeta_p\in\SC$ be a $p$-th primitive root of unity. We have the following ring isomorphisms:
$$\SZ[C_p]\simeq \SZ[X]/(X^p-1)\simeq\SZ[X]/(X-1)\times\SZ[X]/(1+X+\ldots+X^{p-1})\simeq\SZ\times\SZ[\zeta_m].$$
Consequently, 
$$\int_{\SQ}(\SZ[C_p])=\int_{\SQ}(\SZ)\cap\int_{\SQ}(\SZ[\zeta_m])=\int_{\SQ}(\SZ[\zeta_m]).$$
But we know that this last ring is a Pr\"ufer domain (cf. \cite{bib:LW2012} or, for instance, see \cite[Corollary VIII.60]{bib:C2}).
\end{example}

In fact, this example may be easily generalized thanks to the following general result that we recall:
  
\begin{theorem}\cite[Theorem 1]{bib:PW1950}\label{thm:2.6}
If $G$ is a finite abelian group of order $n$ and if the characteristic of the field $K$ does not divides $n$, the group algebra $K[G]$ is iomorphic to a product of rings of the form
$$K[G]\simeq \prod_{d\vert n}(K(\zeta_d))^{a_d}$$ 
where $a_d\in\SN$  and $\zeta_d$ is a $d$-th primitive root of unity. More precisely, 
$$a_d=\frac{\#\{\,g\in G \mid \textrm{ord}(g)=d\,\}}{[\,K(\zeta_d)\,:\,K\,]}.$$
\end{theorem}

\begin{corollary}\label{cor:2.3}
Assume that $K$ is an algebraic number field, with ring of integers $\Oo_K$, and that $G$ is a finite group with exponent $m$. Then, $$\int_K(\Oo_K[G])=\int_{K}(\Oo_K[\zeta_m]).$$
\end{corollary}

\begin{proof}
The $\Oo_K$-module $\Oo_K[G]$ lies inside the $K$-vector space $K[G]$ with basis $\{g\}_{g\in G}$. Thus, it follows from Theorem~\ref{thm:2.6} that $\Oo_K[G])=\prod_{d\vert m}(\Oo_K[\zeta_d])^{a_d}$. Consequently,

\quad\quad\quad $\int_K(\Oo_K[G])=\int_K(\cap_{d\vert m}\int_{K}(\Oo_K[\zeta_d])=\int_{K}(\Oo_K[\zeta_m]).$
\end{proof}

\begin{proposition}
Assume that $K$ is an algebraic number field, with ring of integers $\Oo_K$, and that $G$ is a finite group with exponent $m$. 
The ring $\int_K(\Oo_K[G])$ is a Pr\"ufer domain if and only if $G$ is commutative and $\Oo_K[\zeta_m]$ is integrally closed.
\end{proposition}

\begin{proof}
It follows from Theorem~\ref{th:14} and Corollary~\ref{cor:2.3} that the conditions are necessary. Conversely, the conditions are sufficient: if $\Oo_K[\zeta_m]$ is integrally closed, then $\int_K(\Oo_K[G])=\int_K(\Oo_{K(\zeta_m)})$ which is a Pr\"ufer domain by \cite[Theorem 3.7]{bib:LW2012} or \cite[Corollary VIII.60]{bib:C2}.
\end{proof}

The ring $\Oo_K[\zeta_m]$ is integrally closed if and only if $\Oo_{K(\zeta_m)}= \Oo_K[\zeta_m]$, equivalently, if $\Oo_{K(\zeta_m)}=\Oo_K\cdot\SZ[\zeta_m]$. With respect to this last equality recall:

\begin{proposition}\cite[Lemma 1.2]{bib:shimura}.
Let $M$ and $N$ be two algebraic number fields. If $M$ and $N$ are linearly disjoint and if the discriminant of $MN$ is equal to the product of the discriminants of $M$ and $N$, then $\Oo_{MN}=\Oo_M\cdot\Oo_N$.	
\end{proposition}

\begin{corollary}\label{cor:2.6}
If	the prime divisors of $m$ are not ramified in $K$, then $\Oo_{K(\zeta_m)}=\Oo_K\cdot\SZ[\zeta_m]$, and hence, $\Oo_K[\zeta_m]$ is integrally closed. 
\end{corollary}

Asking Chat GPT about the equality $\Oo_K[\zeta_m]=\Oo_{K(\zeta_m)}$, it suggested both following examples. The first one shows that the condition of Corollary~\ref{cor:2.6} is not necessary.

\begin{example}
Let $K=\SQ(\sqrt{5})$ and $m=5$. Then $\Oo_K=\SZ\left[\frac{1+\sqrt{5}}{2}\right]$ and $\Oo_{K(\sqrt{5})}\cdot\SZ[\zeta_5]=\SZ\left[\frac{1+\sqrt{5}}{2},\zeta_5\right]$.
But $K=\SQ(\sqrt{5})=\SQ(\zeta_5+\zeta_5^{-1})$, 
thus $K(\zeta_5)=\SQ(\zeta_5)$ and $\Oo_{K(\zeta_5)}=\SZ[\zeta_5].$
\end{example}

On the contrary, the second example shows that the ring $\Oo_K[\zeta_m]$ is not always integrally closed.

\begin{example}[suggested to Chat GPT by Math StackExchange]
Let $K=\SQ(\sqrt{2})$ and $m=4$. Then, $\Oo_K=\SZ[\sqrt{2}]$ and $\Oo_{\SQ(\zeta_4)}=\SZ[i]$ while $K(\zeta_m)=\SQ(\sqrt{2},i)=\SQ\left(\frac{1+i}{\sqrt{2}}\right)=\SQ(\zeta_8)$ and $\Oo_{\SQ(\zeta_8)}=\SZ[\zeta_8]\not=\SZ[\sqrt{2},i]$. Thus, $\Oo_K[\zeta_4]$ is not integrally closed.
\end{example}

\noindent{\em Concluding remark}. It could be interesting to determine, if not known, when $$\Oo_K[\zeta_m]=\Oo_{K(\zeta_m)}.$$

\medskip

\noindent {\em Acknowledgment}. I would like to thank Giulio Peruginelli for helping me to correct the first version of this paper.


\end{document}